\documentclass[11pt,a4paper]{article}

\usepackage[utf8]{inputenc}
\usepackage[english]{babel}
\usepackage{amsfonts}
\usepackage{amsmath}
\usepackage{amssymb}
\usepackage{amsthm}
\usepackage{mathrsfs}
\usepackage{mathtools} 
\mathtoolsset{showonlyrefs=true} 
\usepackage{microtype}
\usepackage{marvosym}
\usepackage{wasysym} 
\usepackage{stmaryrd}
\usepackage{dsfont}
\usepackage{enumerate}
\usepackage{multicol,multirow}
\usepackage{graphicx}

\makeatletter
\newdimen\rh@wd
\newdimen\rh@hta
\newdimen\rh@htb
\newbox\rh@box
\def\rh@measure#1{\setbox\rh@box=\hbox{$#1$}\rh@wd=\wd\rh@box \rh@hta=\ht\rh@box}

\def\widecheck#1{\rh@measure{#1}%
  \setbox\rh@box=\hbox{$\widehat{\vrule height \rh@hta width\z@ \kern\rh@wd}$}%
  \rh@htb=\ht\rh@box \advance\rh@htb\rh@hta \advance\rh@htb\p@
  \ooalign{$\vrule height \ht\rh@box width\z@ #1$\cr
           \raise\rh@htb\hbox{\scalebox{1}[-0.95]{\box\rh@box}}\cr}}
\makeatother

\usepackage{xcolor}
\usepackage{authblk}
\usepackage{url}

\usepackage[hidelinks]{hyperref}
\hypersetup{
	colorlinks=false,
	linkcolor=blue,
	urlcolor=red,
	linktoc=all}

\usepackage{tikz}
\usetikzlibrary{decorations.pathreplacing}
\usetikzlibrary{decorations.pathmorphing}
\usetikzlibrary{calc}
\usetikzlibrary{decorations.pathreplacing}
\usetikzlibrary{decorations.pathmorphing}
\usetikzlibrary{arrows}
\usetikzlibrary{patterns,fadings}
\usepackage[all,cmtip]{xy}

\usepackage{titletoc}
\titlecontents{section} 
[2.3em]                 
{\rmfamily}             
{\contentslabel{2.3em}} 
{\hspace*{-2.3em}}
{\titlerule*[1pc]{.}\contentspage}

\newcommand{\oneop}{\mathds{1}}

\newcommand{\rbar}{{\overline{r}}}
\newcommand{\iotabar}{{\overline{\iota}}}
\newcommand{\vtilde}{{\tilde{v}}}
\newcommand{\Ehat}{{\widehat{E}}}
\newcommand{\Echeck}{{\widecheck{E}}}
\newcommand{\Ebar}{{\overline{E}}}

\newcommand{\slot}{{\,\cdot \,}}

\newcommand{\C}{\mathcal{C}}

\newcommand{\cF}{\mathcal{F}}

\newcommand{\E}{\mathcal{E}}

\newcommand{\D}{\mathcal{D}}

\newcommand{\A}{\mathcal{A}}

\newcommand{\cL}{\mathcal{L}}
\renewcommand{\L}{\cL}
\newcommand{\M}{\mathcal{M}}

\newcommand{\N}{\mathcal{N}}

\renewcommand{\L}{\cL}

\newcommand{\RR}{\mathbb{R}}

\newcommand{\CC}{\mathbb{C}}

\newcommand{\NN}{\mathbb{N}}

\DeclareMathOperator{\End}{End}

\DeclareMathOperator{\vNMor}{vNMor}
\DeclareMathOperator{\Ind}{Ind}

\DeclareMathOperator{\Ad}{Ad}

\DeclareMathOperator{\id}{id}

\def\II{{I\!I}}

\newcommand{\lqq}{\lq\lq}

\usepackage{xspace}
\DeclareRobustCommand{\eg}{e.g.\@\xspace}
\DeclareRobustCommand{\cf}{cf.\@\xspace}
\DeclareRobustCommand{\ie}{i.e.\@\xspace}
\DeclareRobustCommand{\p}{p.\@\xspace}
\DeclareRobustCommand{\Sec}{Sec.\@\xspace}
\DeclareRobustCommand{\Prop}{Prop.\@\xspace}
\DeclareRobustCommand{\Lem}{Lem.\@\xspace}
\DeclareRobustCommand{\Cor}{Cor.\@\xspace}
\DeclareRobustCommand{\Thm}{Thm.\@\xspace}
\DeclareRobustCommand{\Ch}{Ch.\@\xspace}

\DeclareRobustCommand{\Ex}{Ex.\@\xspace}

\DeclareRobustCommand{\Def}{Def.\@\xspace}
\DeclareRobustCommand{\Rmk}{Rmk.\@\xspace}
\DeclareRobustCommand{\eq}{eq.\@\xspace}

\makeatletter
\DeclareRobustCommand{\etc}{%
    \@ifnextchar{.}%
        {etc}%
        {etc.\@\xspace}%
}
\makeatother

\newcommand{\Cstar}{$C^\ast$\@\xspace}
\newcommand{\Wstar}{$W^\ast$\@\xspace}

\def\u1net{{\A_\RR}}

\theoremstyle{plain}
\newtheorem{thm}{Theorem}[section]
\newtheorem{cor}[thm]{Corollary}
\newtheorem{lem}[thm]{Lemma}
\newtheorem{prop}[thm]{Proposition}

\theoremstyle{definition}
\newtheorem{defi}[thm]{Definition}

\theoremstyle{remark}

\newtheorem{rmk}[thm]{Remark}

\numberwithin{equation}{section}

\newcommand{\colN}{gray!20}
\newcommand{\colM}{gray!55}

\tikzstyle{shaded}=[fill=red!10!blue!20!gray!30!white]
\tikzstyle{unshaded}=[fill=white]

\newcommand{\tikzmath}[2][]
     {\vcenter{\hbox{
     \begin{tikzpicture}[#1]#2
     \end{tikzpicture}}}
     }

\begin{document}

\title{\LARGE A planar algebraic description of conditional expectations}

\author{\Large Luca Giorgetti}

\affil{\normalsize Dipartimento di Matematica, Universit\`a di Roma Tor Vergata\\

Via della Ricerca Scientifica, 1, I-00133 Roma, Italy\\

{\tt giorgett@mat.uniroma2.it}}

\date{}

\maketitle

\begin{abstract}
Let $\N\subset\M$ be a unital inclusion of arbitrary von Neumann algebras. We give a 2-\Cstar-categorical/planar algebraic description of normal faithful conditional expectations $E:\M\to\N\subset\M$ with finite index and their duals $E':\N'\to\M'\subset\N'$ by means of the solutions of the conjugate equations for the inclusion morphism $\iota:\N\to\M$ and its conjugate morphism $\iotabar:\M\to\N$. In particular, the theory of index for conditional expectations admits a 2-\Cstar-categorical formulation in full generality. Moreover, we show that a pair $(\N\subset\M, E)$ as above can be described by a Q-system, and vice versa.
These results are due to Longo in the subfactor/simple tensor unit case \cite[\Thm 5.2]{Lon90}, \cite[\Thm 5.1]{Lon94}.
\end{abstract}

\tableofcontents

\vfill

{\footnotesize Supported by the European Union's Horizon 2020 research and innovation programme H2020-MSCA-IF-2017 under Grant Agreement 795151 \emph{Beyond Rationality in Algebraic CFT: mathematical structures and models} and by the MIUR Excellence Department Project awarded to the Department of Mathematics of the University of Rome Tor Vergata, CUP E83C18000100006.}

\newpage

\section{Introduction}

In \cite{Jon83}, Jones introduced the notion of index for subfactors, a number that measures the \lqq relative size" of a factor with respect to another factor.
Recall that a factor is a von Neumann algebra with trivial center, typically infinite-dimensional as a complex algebra, and that a subfactor consists of two factors, one included in the other with the same unit. 
In the simplest case of subfactors coming from finite group actions, \eg by taking crossed product extensions or fixed point subalgebras, the index coincides with the cardinality of the group. If one considers also intermediate subfactors, \ie group-subgroup subfactors, the index equals the group-theoretical one. The striking result proven in \cite{Jon83} is that the index of a subfactor is bound to the set $\{4 \cos^2(\pi/k), k=3,4,5,\ldots\} \cup [4,\infty]$, hence \lqq quantized" between $1$ and $4$, and that every value in the set is realized.

Initially, the theory of index has been developed for $\II_1$ subfactors $\N\subset\M$, \ie assuming $\M$ to be endowed with a normal faithful tracial state. See \cite{JoSuBook}, \cite{EvKaBook} for the background. In this case, the index $[\M:\N]$ of $\N$ in $\M$ can be defined as $\dim_\N(L^2\M)$ using von Neumann's notion of dimension of the standard representation $L^2\M$ seen as an $\N$-module. The index admits also a variational characterization \cite{PiPo86} as follows. Let $\tau$ be the trace on $\M$ and denote by $E^\tau$ the unique trace-preserving ($\tau = \tau \circ E^\tau$) conditional expectation of $\M$ onto $\N$. By \cite{PiPo86}, the number $\lambda = [\M:\N]^{-1}$ is the best possible $\lambda\geq 0$ such that $E^\tau(x) \geq \lambda x$ for every positive $x\in\M$. Subsequently, in \cite{Kos86} and \cite{Lon89}, the notion of index has been extended to arbitrary subfactors $\N\subset\M$ (not necessarily tracial) endowed with a normal faithful conditional expectation $E$. This index is a number which depends on $\N\subset\M$ and $E$ and it is again quantized between $1$ and $4$. It is characterized as the inverse of the best constant $\lambda\geq 0$ such that $E(x) \geq \lambda x$ for every positive $x\in\M$, provided that $\N$ and $\M$ are not finite-dimensional factors (full matrix algebras). Denoted this number by $\Ind(E)$, it holds that $\Ind(E^\tau) = [\M:\N]$ if $\M$ is tracial and $E^\tau$ is the trace-preserving conditional expectation as before, thus recovering Jones' original definition for $\II_1$ subfactors. 

In the absence of a trace and assuming that the set $\E(\M,\N)$ of normal faithful conditional expectations of $\M$ onto $\N$ is non-empty, one can look for expectations that minimize the index function. This was first done in \cite{Hia88}, \cite{Lon89}, \cite{Hav90}.
In the case of factors, there is a unique such expectation, called minimal expectation and denoted by $E^0$. The number $\Ind(E^0) = [\M:\N]_0$ is called the minimal index of the subfactor. If $\M$ is tracial, hence $[\M:\N]$ can be computed, and if $\N\subset\M$ is irreducible, namely $\N'\cap\M = \CC\oneop$, there is a unique element in $\E(\M,\N)$, hence $E^\tau = E^0$. In particular, $[\M:\N] = [\M:\N]_0$. This latter condition is equivalent to the notion of extremality for $\II_1$ subfactors \cite{PiPo86}, \cite{PiPo91}, and it is implied \eg if the subfactor has finite depth \cite{Pop90}.

There are generalizations of index for inclusions of von Neumann algebras with non-trivial centers, both in the tracial and in the non-tracial case. In the case of finite-dimensional centers, the index becomes a matrix (with finite size, and entry-wise finite if and only if the inclusion has finite index) which is either defined using the (unique Markov) trace \cite{GHJ89}, \cite{Jol90}, \cite{BCEGP20}, or using the (unique) minimal expectation \cite{Ter92}, \cite{Hav90}, \cite{BDH14}, \cite{GiLo19}, \cite{Gio19}. In both cases, one can consider canonical notions of scalar-valued Jones/minimal index as well. For inclusions of algebras with arbitrary (atomic or diffuse) centers, the tracial index has been studied \eg in \cite{Jol90}, \cite{Jol91} and the minimal index in \cite{FiIs96}. In \cite{FiIs96}, explicit examples are given where minimal expectations, which always exist for arbitrary inclusions, are no longer unique if the centers are infinite-dimensional. Further and very beautiful developments on the theory of index (on the line of Kosaki's generalization \cite{Kos86} but for conditional expectations on inclusions of arbitrary von Neumann algebras) are due to \cite{BDH88} and \cite{PopBook}. We also mention Watatani's definition of index for conditional expectations between \Cstar-algebras \cite{Wat90}, which is close in spirit to \cite{BDH88} in the strongly finite index case.

From a different perspective, in \cite{Lon89}, \cite{Lon90}, Longo discovered a connection between the minimal index for subfactors and the statistical dimension of certain superselection sectors (describing \lqq localized" charged particle states) in the Algebraic formulation of Quantum Field Theory \cite{DHR71}, \cite{DHR74}. 
See \cite{HaagBook} for an introduction. The statistical dimension, almost by its very definition, can be made categorical, \cf \cite[\Sec III]{DHR74}.
Namely, it makes sense for an abstract tensor \Cstar-category (with simple tensor unit), not necessarily the tensor \Cstar-category of localized superselection sectors in AQFT. This was first achieved in \cite{DoRo89}, and then more generally in \cite{LoRo97}, and it gives a tensor \Cstar-categorical description of the minimal index for subfactors and of the statistical dimension, thus also called intrinsic/tensor \Cstar-categorical dimension. 
A generalization of this theory of dimension for multitensor \Cstar-categories and 2-\Cstar-categories with finitely decomposable tensor units appears in \cite{GiLo19}.

Following yet another different path, the theory of Jones index for $\II_1$ subfactors (assuming hyperfinite and finite depth) was soon shown to admit a categorical formulation as well. In \cite{Pop90}, Popa proved that finite index and finite depth (or more generally amenable \cite{Pop94}) hyperfinite $\II_1$ subfactors can be classified by \lqq simpler" combinatorial data: a collection of four finite-dimensional algebras fulfilling the commuting square condition. The information contained in the commuting square can be axiomatized in various equivalent ways: Ocneanu's paragroups \cite{Ocn88}, Popa's $\lambda$-lattices \cite{Pop95}, Jones' planar algebras \cite{Jon99}, or more categorically using Connes' bimodules \cite{Bis97} as the standard invariant of the subfactor. The standard invariant of $\N\subset\M$ is the rigid multitensor \Cstar-category generated by $L^2\M$ seen as an $\N\oplus\M$ bimodule, or equivalently the unitary 2-shaded planar algebra associated with $\N\subset\M$. See \cite{MPS10}, \cite{Gho11}, \cite{BHP12} and \cite{BCEGP20}.
The index of the subfactor is then encoded by the loop parameter of the planar algebra, a scalar in this case.

In this work, we provide a 2-\Cstar-categorical/planar algebraic description of the Jones--Kosaki theory of index for conditional expectations, not necessarily minimal or trace-preserving, between arbitrary properly infinite von Neumann algebras, not necessarily with trivial or finite-dimensional centers (Theorem \ref{thm:EEbarandconjeqns}). As a consequence, we show that an arbitrary unital inclusion of properly infinite von Neumann algebras $\N\subset\M$ (equipped with a given finite index conditional expectation $E\in\E(\M,\N)$) is described by a Q-system, a notion introduced by Longo in the context of infinite subfactors with finite index \cite{Lon94}. Vice versa, every Q-system in $\End(\N)$ defines an extension $\M$ with a given finite index conditional expectation $E$ (Theorem \ref{thm:Qsys}). By a Q-system (Definition \ref{def:Qsys}), we mean in this paper a unitary Frobenius algebra object (Definition \ref{def:uFralgebraobject}) in $\End(\N)$ (the monoidal \Cstar-category of the endomorphisms of $\N$) as in \cite{Lon94}, see also \cite{LoRo97}, with an additional invertibility condition on the (co)unit. The invertibility condition is always fulfilled if $\N$ is a factor, \ie if $\End(\N)$ has simple tensor unit. Hence Definition \ref{def:Qsys} boils down to the original definition of Q-system in that case. We refer the reader to Remark \ref{rmk:invertcond} for references to earlier appearances of this invertibility condition in the literature. 
To our knowledge, variations or special cases of our main results: Theorem \ref{thm:EEbarandconjeqns} and \ref{thm:Qsys}, both originally due to Longo in the subfactor case, appear in \cite{Lon90}, \cite{Lon94}, \cite{FiIs95}, \cite{Mue03-I}, \cite{BDH14}, \cite{BKLR15}, \cite{GiYu20}, \cite{CHPJP21}.

The paper is organized as follows. In Section \ref{sec:conjin2Cstarcats}, we recall the definition of 2-\Cstar/\Wstar-category and the notion of conjugate 1-morphisms (also called adjoint or dual 1-morphisms) in that context, which is given by means of the solutions of the conjugate equations \cite{LoRo97} (also called zig-zag or snake or adjoint or duality equations).

In Section \ref{sec:condexp}, we review Kosaki's definition of index for a conditional expectation \cite{Kos86}, which is in general not just a positive scalar or infinite. Let $\N\subset\M$ be a unital inclusion of arbitrary von Neumann algebras. Given $E\in\E(\M,\N)$, the index $\Ind(E)$ is an element of the extended positive part (as defined by Haagerup \cite{Haa79I}) of $Z(\M)$, the center of $\M$. If $\Ind(E)\in Z(\M)$, in which case $\Ind(E)\geq \oneop$ in the sense of operators, we say that $E$ has finite index, and infinite otherwise.
In the case of subfactors, \ie $Z(\N) = Z(\M) = \CC\oneop$, then $\Ind(E) = \lambda \oneop$ for some scalar $\lambda\in\RR$, $\lambda \geq 1$, or infinite. 

In Section \ref{sec:iterateddualexp}, we study the dual expectation $E'\in\E(\N',\M')$ of a finite index conditional expectation $E\in\E(\M,\N)$, obtained by normalizing the Kosaki--Haagerup dual operator-valued weight of $E$.
We compute the index of $E'$, namely we show that $\Ind(E') = E(\Ind(E))\in Z(\M)$ (Lemma \ref{lem:Eprimeinverse}), and we observe that the bidual expectation $E''\in\E(\M,\N)$ does not always coincide with $E$. It is the case that $E=E''$ if and only if $E(\Ind(E)) = \Ind(E)$ (Proposition \ref{prop:bidual}), namely if and only if $\Ind(E) \in Z(\M)\cap Z(\N)$. This condition is always satisfied \eg if $\M$ is a factor, or if $E$ has scalar index. We provide a formula for the index of iterated dual expectations beyond $E''$ (Proposition \ref{prop:interateddualexp}), whose index sits either in $Z(\M)$ or in $Z(\N)$ depending on the parity of the iteration, and we speculate on the convergence of the sequence of indices.

In Section \ref{sec:dualexp}, we recall the definition of Longo's canonical endomorphism \cite{Lon87} and its relation to the Jones tower/tunnel associated with $\N\subset\M$. We transport the dual expectation $E'$ one step up and one step down in the tower/tunnel.

In Section \ref{sec:expareconjeqnssols}, we provide a representation formula for finite index conditional expectations $E\in\E(\M,\N)$, together with their dual expectation $E'\in\E(\N',\M')$, or better their conjugate expectation $\Ebar\in\E(\N,\iotabar(\M))$, by means of the solutions of the conjugate equations for the inclusion morphisms $\iota:\N\to\M$ and its conjugate $\iotabar:\M\to\N$ (Theorem \ref{thm:EEbarandconjeqns}). In particular, this says that $\iota$ and $\iotabar$ are conjugate 1-morphisms in the 2-\Cstar-categorical sense of Section 2 if and only if there exists an expectation $E$ with finite index, \ie if the inclusion $\N\subset\M$ has finite index (Definition \ref{def:finindex}). Moreover, $\Ind(E)\in Z(\M)$ can be expressed using the solutions associated with $E$ as a double loop diagram (Corollary \ref{cor:Indexsolconjeqns}),
and one recovers the standard solutions \cite{LoRo97}, \cite{GiLo19} of the conjugate equations by looking at the minimal expectation $E=E^0$ (Corollary \ref{cor:stdsol}).
The roles played by $E$ and $\Ebar$ are not symmetric in general (Proposition \ref{prop:Ebarbar}), while the roles played by $\iota$ and $\iotabar$ are, namely $\overline{\iotabar} \cong \iota$.

In Section \ref{sec:Qsys}, we call a Q-system a unitary (or \Cstar) Frobenius algebra in $\End(\N)$ with an additional invertibility condition (Definition \ref{def:Qsys}). This condition is always satisfied \eg if $\N$ is a factor, \ie if $\End(\N)$ has simple tensor unit. We show that every unital inclusion of von Neumann algebras $\N\subset\M$ with a given finite index conditional expectation $E\in\E(\M,\N)$ defines a Q-system. Vice versa, every Q-system in $\End(\N)$ defines such a pair $(\N\subset\M,E)$ (Theorem \ref{thm:Qsys}).

\section{Conjugation in 2-\Cstar-categories}\label{sec:conjin2Cstarcats}

A (strict) 2-category $\C$ \cite[\Ch XII]{Mac98}, \cite{JhYaBook} is called a \textbf{2-\Cstar-category} \cite{LoRo97} if every local category $\C(X,Y)$ is a \Cstar-category for every objects $X,Y\in\C$, and if the horizontal composition of 2-morphisms $t\in\C(X,Y)(\alpha,\beta)$, $s\in\C(Y,Z)(\gamma,\delta)$ denoted by
\footnote{Thinking of $\otimes$ as the composition of maps $\alpha:X\to Y$, $\beta:Y\to Z$, $\beta\otimes\alpha = \beta \circ \alpha: X \to Z$ or diagrammatically on the plane with horizontal arrows pointing from right to left.} 
$s\otimes t\in\C(X,Z)(\gamma\otimes\alpha,\delta\otimes\beta)$ is bilinear and preserves the involution, \ie $(s\otimes t)^* = s^* \otimes t^*$ holds in $\C(X,Z)(\delta\otimes\beta,\gamma\otimes\alpha)$ for every 1-morphisms $\alpha,\beta\in\C(X,Y)$, $\gamma,\delta\in\C(Y,Z)$. 

Recall from \cite{GLR85} that a \Cstar-category is a category $\D$ such that the morphisms form complex Banach spaces with respect to a norm $\|\cdot\|$, the composition of morphisms $t\in\D(\alpha,\beta)$, $r\in\D(\beta,\gamma)$ denoted by $r t \in \D(\alpha,\gamma)$ is bilinear and fulfills $\|rt\| \leq \|r\|\|t\|$, and there is an antilinear map $t\in\D(\alpha,\beta) \mapsto t^*\in\D(\beta,\alpha)$ such that $t^{**} = t$, $(rt)^* = t^*r^*$, $t^*t \geq 0$ in $\D(\alpha,\alpha)$ ($t^*t = s^*s$ for some $s\in\D(\alpha,\alpha)$), $1_\alpha^* = 1_\alpha$ and such that the \Cstar identity holds: $\|t^*t\| = \|t\|^2$.

We shall not need in the sequel the notion of direct sum of objects or 1-morphisms in a \Cstar-category or 2-\Cstar-category, hence we do not include it in the definition.
We also write $\alpha:X\to Y$ for 1-morphisms $\alpha\in\C(X,Y)$, and $t:\alpha\Rightarrow\beta:X\to Y$, or just $t:\alpha\Rightarrow\beta$, for 2-morphisms $t\in\C(X,Y)(\alpha,\beta)$ in a 2-category $\C$. Denote by $\id_X:X\to X$ the identity 1-morphism on $X$ and by $1_\alpha:\alpha\Rightarrow\alpha$ the identity 2-morphism on $\alpha$.

\begin{rmk}
A \Cstar-category with one object is a unital \Cstar-algebra. A 2-\Cstar-category with one object is a (strict) monoidal \Cstar-category (sometimes also called tensor under additional assumptions: rigidity and simplicity of the tensor unit) \cite{LoRo97}, \cite{DoRo89}, \cite{EGNO15}.
\end{rmk} 

A pair of 1-morphisms in a 2-\Cstar-category $\C$, denoted by $\iota:X\to Y$ and $\iotabar:Y \to X$ for the sake of uniformity with the following sections, are called \textbf{conjugate} (sometimes also called dual or adjoint) if there is a pair of 2-morphisms denoted by $r: \id_X \Rightarrow \iotabar\otimes\iota : X\to X$ and $\rbar: \id_Y \Rightarrow \iota\otimes\iotabar : Y\to Y$ solving the \textbf{conjugate equations}:
 \begin{align}\label{eq:catconjeqns}
(\rbar^* \otimes 1_\iota) (1_\iota \otimes r) = 1_\iota , \quad (r^* \otimes 1_\iotabar) (1_\iotabar \otimes \rbar) = 1_\iotabar.
\end{align}

We introduce now the 2-\Cstar-category of von Neumann algebras and morphisms. Denote by $\N$, $\M$, $\L$ von Neumann algebras (always assumed with separable predual) and by $\oneop_\N$ the unit of $\N$, or by $\oneop$ when there is no ambiguity. Denote by $\iota: \N \to \M$ a normal injective unital *-homomorphism, below called morphism for short. Denote by $t:\iota_1 \Rightarrow \iota_2$ an intertwiner between two morphisms $\iota_1,\iota_2:\N\to\M$, \ie an operator $t\in\M$ such that $t \iota_1(n) = \iota_2(n) t$ for every $n\in\N$. 
The collection of all von Neumann algebras as objects, morphisms as 1-morphisms, and intertwiners as 2-morphisms forms a 2-\Cstar-category denoted by $\vNMor$. The horizontal composition of morphisms $\alpha:\N\to\M$ and $\beta:\M\to\L$ is given by the composition of maps: $\beta \otimes \alpha: \N\to\L$ defined by $\beta \otimes \alpha := \beta\circ\alpha$, and it is defined accordingly on intertwiners. The involutions and \Cstar-norms in the local categories $\vNMor(\N,\M)$ are the Hilbert space adjoint and the operator norm. This 2-\Cstar-category is also a 2-\Wstar-category \cite{GLR85}, \cite{CHPJP21}, namely the 2-morphism \Cstar-algebras $\vNMor(\N,\M)(\alpha,\beta)$ are also \Wstar-algebras (concrete von Neumann algebras in this case) and the horizontal composition $s\otimes t$ is separately normal (ultraweakly continuous). Several families of abstract 2-\Cstar/\Wstar-categories (semisimple and rigid, \ie admitting a conjugate for every 1-morphism) can in fact be realized in this way on von Neumann algebras. See \cite{HaYa00}, \cite{Yam03}, \cite{BHP12}, \cite{GiYu19}, \cite{HePe20}, \cite{GiYu20}, \cite{BCEGP20}.

The 2-\Cstar-subcategory $\End(\N) = \vNMor(\N,\N)$ with one object $\N$ is a concrete tensor \Cstar-category, and also a \Wstar-category. 
By definition, $Z(\N) := \N'\cap\N = \End(\N)(\id_\N,\id_\N)$, where $\id_{\N}$ is the identity endomorphism of $\N$, \ie the tensor unit of $\End(\N)$.
Note also that $\N'\cap\M = \vNMor(\N,\M)(\iota,\iota)$, if $\N\subset\M$ is a unital inclusion and $\iota:\N\to\M$ is the inclusion morphism.
In particular, $\id_{\N}$ is simple, \ie $\End(\N)(\id_\N,\id_\N) = \CC\oneop$,
if and only if $\N$ is a factor.

\section{Conditional expectations with finite index}\label{sec:condexp}

In this section, we review the theory of index for conditional expectations due to Kosaki \cite{Kos86}, initially considered for subfactors and later further studied in the case of arbitrary inclusions \cite{BDH88}, \cite{PopBook}.
Every unital inclusion of von Neumann algebras with separable predual can be written as $\iota(\N)\subset\M$, with $\iota:\N\to\M$ a morphism in $\vNMor$.

\begin{defi}
A \textbf{conditional expectation} $E$ of $\M$ onto $\iota(\N)$ is a unital completely positive map $E:\M\to\M$ such that $E(\M) = \iota(\N)$ and which is $\iota(\N)$-bimodular, \ie $E(\iota(n_1)m\iota(n_2)) = \iota(n_1)E(m)\iota(n_2)$ for every $m\in\M$, $n_1,n_2\in\N$. See \cite{Sto97} for an overview.
We shall also write $E:\M\to\iota(\N)\subset\M$.
\end{defi}

A conditional expectation $E$ of $\M$ onto $\iota(\N)$ is called normal if it is continuous in the ultraweak operator topology of $\M$, and faithful if $E(m^*m) = 0$ for $m\in\M$ implies $m = 0$.

\begin{defi}
Let $\E(\M,\iota(\N))$ be the set of normal faithful conditional expectations of $\M$ onto $\iota(\N)$.
\end{defi}

Following Kosaki \cite{Kos86}, one can consider the \textbf{index} of an expectation $E\in \E(\M,\iota(\N))$. We recall its definition. Consider the inclusion of commutants $\M' \subset \iota(\N)'$ and denote by $E^{-1}$ the normal faithful semifinite operator-valued weight from $\iota(\N)'$ onto $\M'$ characterized by the following equality of spatial derivatives \cite{Haa79II}, \cite{Con80}:
\begin{align}\label{eq:spderivative}
\frac{d \phi \circ E}{d \psi} = \frac{d \phi}{d \psi \circ E^{-1}},
\end{align}
where $\phi$ and $\psi$ are normal faithful semifinite weights on $\iota(\N)$ and $\M'$, respectively. The operator-valued weight $E^{-1}$ depends only on $E$, not on the chosen weights $\phi$ and $\psi$.
Moreover, $E^{-1}$ is never unital (besides when $\iota(\N) = \M$). The index of $E$ is defined by
$$\Ind(E) := E^{-1}(\oneop).$$
$\Ind(E)$ is in general an element in the extended positive part of $Z(\M)$ \cite{Haa79I} and it does not depend on the Hilbert space representation of $\M$ by the same proof of \cite[\Thm 2.2]{Kos86}. When $\Ind(E)$ is finite, \ie when it is an actual positive element in $Z(\M)$, it is invertible and $\Ind(E)\geq \oneop$. If $\M$ is a factor, as considered in \cite[\Sec 2]{Kos86}, $\Ind(E) = \lambda \oneop$ for some $\lambda \geq 1$.

\begin{defi}\label{def:finindex}
An inclusion $\iota(\N)\subset\M$ is said to have \textbf{finite index} if it admits an expectation $E\in \E(\M,\iota(\N))$ with finite index.  
\end{defi}

\begin{rmk}
For finite index inclusions, by \cite[\Cor 3.18, 3.19]{BDH88}, the finite-dimensionality of $Z(\iota(\N))$, $Z(\M)$ or $\iota(\N)'\cap\M$ are equivalent conditions. See also \cite[\Prop 8.16]{GiLo19} for a 2-\Cstar-categorical proof of this equivalence. In this case, by \cite[\Thm 6.6]{Haa79II}, either all expectations in $\E(\M,\iota(\N))$ have finite index, or none of them has. In the case of infinite-dimensional centers this is no longer true, namely a finite index inclusion may admit normal faithful expectations with infinite index as well.
\end{rmk}

\section{Iterated dual expectations}\label{sec:iterateddualexp}

If $E\in\E(\M,\iota(\N))$ has finite index, which we recall means that $E^{-1}$ is bounded on $\iota(\N)'$ hence defined on $\oneop$, then $E^{-1}$ can be normalized to a normal faithful conditional expectation:
$$E' := \frac{1}{\Ind(E)} E^{-1}.$$ 
Here $1/\Ind(E)$, later also written as $\Ind(E)^{-1}$, denotes the inverse of $\Ind(E)$ in $Z(\M)$.

\begin{defi}
We call $E' \in \E(\iota(\N)',\M')$ the \textbf{dual expectation} of $E$. 
\end{defi}

The operator-valued weight $(E')^{-1}$ from $\M$ onto $\iota(\N)$ defined as in \eqref{eq:spderivative} is also bounded and it has the following easy expression:

\begin{lem}\label{lem:Eprimeinverse}
If $E\in \E(\M,\iota(\N))$ has finite index, then $(E')^{-1} = E(\Ind(E)\slot)$ holds on $\M$. In particular, $E'$ has finite index and $\Ind(E') = E(\Ind(E)) \in Z(\iota(\N))$.
\end{lem}

\begin{proof}
Choose normal faithful semifinite weights $\phi$ and $\psi$ on $\M'$ and $\iota(\N)$, respectively, and set $\widetilde\phi := \phi(\Ind(E) \slot) = \phi(\Ind(E)^{1/2} \slot \Ind(E)^{1/2})$ on $\M'$. 
Then
$$\frac{d \widetilde\phi}{d \psi \circ (E')^{-1}} = \frac{d \widetilde\phi \circ E'}{d \psi} = \frac{d \phi \circ E^{-1}}{d \psi} = \frac{d \phi}{d \psi \circ E}$$
because $(E^{-1})^{-1} = E$ \cite{Kos86}. By \cite[\Prop 8]{Con80}, the left hand side of the above equation is equal to $(\Ind(E)^{1/2}) (d \phi/d \psi \circ (E')^{-1}) (\Ind(E)^{1/2})$. Hence taking the inverse of the previous equality, by \cite[\Thm 9]{Con80} and \cite[\Prop 8]{Con80}, we have the desired formula for $(E')^{-1}$. 
\end{proof}

\begin{defi}
We consider the \textbf{bidual expectation} $E'' \in \E(\M,\iota(\N))$ of $E$ obtained by normalizing $(E')^{-1}$:
$$E'' := \frac{1}{E(\Ind(E))} E(\Ind(E) \slot).$$
\end{defi}

If $\M$ is not a factor or $\Ind(E)$ is not a scalar, the bidual expectation $E''$ need \emph{not} coincide with $E$. We have the following characterization:

\begin{prop}\label{prop:bidual}
Let $E\in \E(\M,\iota(\N))$ with finite index. Then $E = E''$ if and only if $E(\Ind(E)) = \Ind(E)$.
\end{prop}

\begin{proof}
$E = E''$ is equivalent to $E(\Ind(E)m) = E(\Ind(E)) E(m)$ for every $m\in\M$, \ie $\Ind(E)$ is in the multiplicative domain of $E$. Thus the proof follows from Lemma \ref{lem:multdomainE}.
\end{proof}

\begin{lem}\label{lem:multdomainE}
Let $E\in \E(\M,\iota(\N))$. Then the multiplicative domain of $E$ \emph{\cite{Cho74}} denoted by $M_E := \{x\in\M: E(xm) = E(x)E(m) \text{ for every }m\in\M\}$ coincides with the range $\iota(\N)$. 
\end{lem}

\begin{proof}
Clearly $\iota(\N) \subset M_E$. Conversely, if $x\in M_E$ then $E((x-E(x))^*(x-E(x))) = E(x-E(x))^*E(x-E(x)) = 0$ because $E^2 = E$. Thus $x = E(x)$ by faithfulness of $E$. 
\footnote{We thank Jesse Peterson for providing this short proof that we could not find in the literature.}
\end{proof}

\begin{rmk}
The equality $E(\Ind(E)) = \Ind(E)$ is equivalent to $\Ind(E)\in Z(\M)\cap Z(\iota(\N))$ and it is not always satisfied. Indeed, let $\iota:\N\to\M$ be such that the inclusion $\iota(\N)\subset\M$ is \emph{connected} in the terminology of \cite{GHJ89}, namely $Z(\iota(\N))\cap Z(\M) = \CC\oneop$.
There exist conditional expectations in $\E(\M,\iota(\N))$ with finite non-scalar index, as one can easily construct \eg in the case of $\N$ and $\M$ with finite-dimensional centers.
\end{rmk}

Iterating the process of taking dual expectations, let $E^{(n)}$ be the $n$-th dual expectation of $E$ for $n\in\NN$. In particular, $E^{(0)} = E$, $E^{(1)} = E'$, $E^{(2)} = E''$. Thus $E^{(n)}$ is either in $\E(\M,\iota(\N))$ or in $\E(\iota(\N)',\M')$, depending on the parity of $n$. 
By arguing as in the proof of Lemma \ref{lem:Eprimeinverse}, one can compute the expressions of the operator-valued weights $(E^{(n)})^{-1}$, thus the indices of every $E^{(n)}$:

\begin{prop}\label{prop:interateddualexp}
Let $E\in \E(\M,\iota(\N))$ with finite index. Then $\Ind(E'') = E'(E(\Ind(E)))$, $\Ind(E''') = E''(E'(E(\Ind(E))))$. More generally, for $n\geq 3$, 
$$\Ind(E^{(n)}) = E^{(n-1)}(\cdots E^{(1)}(E(\Ind(E)))),$$
or iteratively
$$\Ind(E^{(n)}) = E^{(n-1)}(\Ind(E^{(n-1)})),$$
where $\Ind(E^{(n)})$ is either in $Z(\M)$ or in $Z(\iota(\N))$ depending on the parity of $n$.
\end{prop}

Assume for the moment that $\iota(\N) \subset \M$ is connected, namely $Z(\iota(\N))\cap Z(\M) = \CC\oneop$. This assumption is not very restrictive in the sense that every inclusion with a given conditional expectation can be decomposed as a direct sum or direct integral of connected inclusions \cite[\Rmk 3.4]{Ter92}, \cite[\Thm 1]{FiIs96}. If for some $E\in\E(\M,\iota(\N))$ the sequence of operators $\Ind(E^{(n)})$ happens to converge, \eg in the strong operator topology, then the limit belongs to $Z(\M) \cap Z(\iota(\N))$, hence it is a scalar multiple of $\oneop$. If we denote this scalar by $[E]$, then $\lim_n \Ind(E^{(n)}) = [E] \oneop$, $[E] \geq 1$, and one can regard it as a scalar notion of index for $E$. 

We don't know wether the sequence $\Ind(E^{(n)})$ converges for an arbitrary $E\in\E(\M,\iota(\N))$. It is bounded and positive but not monotone in general, as one can check in the case of finite-dimensional centers by using \cite[\Thm 2.5]{Hav90}, \cite[\Prop 2.3]{Ter92}.

\section{Dual expectation and the Jones tower}\label{sec:dualexp}

In this section, we review some definitions and we transport the dual expectation $E'$ of a finite index expectation $E\in\E(\M,\iota(\N))$ one step down and one step up in the Jones tower.

From now on we assume that $\N$ and $\M$ are \emph{properly infinite} von Neumann algebras.

\begin{rmk}
If $\N$ and $\M$ are not properly infinite, tensor both algebras with a type $I_\infty$ factor $\cF$ with separable predual, and replace $\iota$ and $E$ respectively with $\iota \otimes \id_\cF$ and $E \otimes \id_\cF$. By a result of Combes and Delaroche \cite[\Lem 2.3]{CoDe75}, every conditional expectation of $\M\otimes\cF$ onto $\iota(\N)\otimes\cF$ is of this form. Namely, $\E(\M\otimes\cF,\iota(\N)\otimes\cF) = \E(\M,\iota(\N))\otimes\id_\cF$. 
Moreover, by the same arguments leading to \cite[\Prop 3.6, \Lem 3.7]{Tsu91}, $\Ind(E \otimes \id_\cF) = \Ind(E) \otimes 1_\cF$.
\end{rmk}

Let $\xi$ be a \emph{jointly} cyclic and separating vector for $\iota(\N)$ and $\M$, whose existence is guaranteed in the properly infinite case by \cite{DiMa71}. Denote by $J_{\iota(\N),\xi}$ and $J_{\M,\xi}$ the respective modular conjugations. Let $j_{\iota(\N)} := \Ad{J_{\iota(\N),\xi}}$ and $j_{\M} := \Ad{J_{\M,\xi}}$ be their adjoint actions, and let $\gamma := j_{\iota(\N)} \circ j_{\M} : \M \to \M$ be the \textbf{canonical endomorphism} of $\M$ defined by the inclusion $\iota(\N) \subset\M$ and by the vector $\xi$ \cite{Lon87}. Note that $\gamma$ depends on $\xi$ only up to conjugation with a unitary in $\iota(\N)$ \cite[\Sec 1]{Lon87}. By definition, 
\begin{align}\label{eq:Jonestower}
\gamma(\M) \subset \iota(\N) \subset \M \subset \M_1
\end{align}
is the beginning of the Jones tower \cite{Jon83}, where $\M_1 := j_{\M}(\iota(\N)')$.
The Jones extension $\M_1$ coincides with the von Neumann algebra generated by $\M$ and by the Jones projection of $E$. The canonical endomorphism can be defined at every level of the Jones tower and it provides spatial isomorphisms two steps up and two steps down in the tower. See \cite[\Sec 2.5]{LoRe95} and \cite[\Sec 2.1, 2.2]{BDG21-online} for a review.

By composing with the modular conjugations, one can view the dual expectation $E'$ either as an element in $\E(\M_1,\M)$, by setting $\Ehat := j_\M \circ E' \circ j_\M$, or as an element in $\E(\iota(\N),\gamma(\M))$, by setting $\Echeck := j_{\iota(\N)} \circ E' \circ j_{\iota(\N)}$. We rewrite \eqref{eq:Jonestower} with $\Ehat$ and $\Echeck$:
\begin{align}
\gamma(\M) \stackrel{\scalebox{0.8}{\Echeck}}{\subset} \iota(\N) \stackrel{E}{\subset} \M \stackrel{\Ehat}{\subset} \M_1.
\end{align}

\begin{lem}\label{lem:IndEcheckEhat}
It holds $\Ind(\Ehat) = j_{\M}(\Ind(E')) \in Z(\M_1) $ and $\Ind(\Echeck) = \Ind(E') \in Z(\iota(\N))$.
\end{lem}

\begin{proof}
It follows by using \cite[\Lem 1.3]{Kos86}, Lemma \ref{lem:Eprimeinverse} and the fact that $j_{\iota(\N)}$ acts trivially on positive elements in $Z(\iota(\N))$.
\end{proof}

\begin{lem}\label{lem:Eprimeprime}
In the above notation, 
$$(\Echeck)^\wedge = E'', \quad (\Ehat)^\vee = E'',$$
where $E''$ is the bidual expectation.
\end{lem}

\begin{proof}
Immediate from $(\Echeck)' = j_{\iota(\N)} \circ E'' \circ j_{\iota(\N)}$ and $(\Ehat)' = j_{\M} \circ E'' \circ j_{\M}$, which can be checked as in Lemma \ref{lem:IndEcheckEhat}. 
\end{proof}

\section{Expectations are solutions of the conjugate equations}\label{sec:expareconjeqnssols}

In this section, we show that finite index expectations and their duals correspond to the solutions of the conjugate equations for the inclusion morphism and its conjugate morphism. In particular, this gives a 2-\Cstar-categorical description of the Jones--Kosaki index for conditional expectations \cite{Jon83}, \cite{Kos86}, \cite{BDH88}, \cite{PopBook} between arbitrary von Neumann algebras. 

Assume that $\N$ and $\M$ are properly infinite von Neumann algebras in standard form. Following \cite[\Sec 3]{Lon90}, we call $\iotabar : \M \to \N$ a \textbf{conjugate morphism} of $\iota:\N\to\M$ in the sense of Connes' bimodules, not necessarily in the 2-\Cstar-categorical sense of Section \ref{sec:conjin2Cstarcats}, if 
$$\iotabar := \iota^{-1} \circ \gamma,$$ 
where $\gamma$ is a canonical endomorphism of $\M$ defined by $\iota(\N) \subset\M$ as in the previous section. The formula for $\iotabar$ makes sense because $\gamma(\M) \subset \iota(\N)$. Note that if $\iota(\N)=\M$, then $\gamma = \id_\M$ and $\iotabar = \iota^{-1}$. 
In general, $\overline{\iotabar} \cong \iota$ (unitary 2-isomorphism) see \eg \cite[\Sec 2.2]{Lon18}. By definition, 
\begin{align}\label{eq:gammaisiotaiotabar}
\gamma = \iota \circ \iotabar.
\end{align}

\begin{defi}
Let $E\in\E(\M,\iota(\N))$ with finite index. In the notation of the previous section, define $\Ebar := \iota^{-1} \circ \Echeck \circ \iota$, then $\Ebar \in \E(\N,\iotabar(\M))$. We call $\Ebar$ a \textbf{conjugate expectation} of $E$ because it is associated with a conjugate morphism $\iotabar:\M\to\N$.
\end{defi}

\begin{rmk}\label{rmk:IndEbar}
It holds $\Ind(\Ebar) = \iota^{-1}(\Ind(E')) \in Z(\N)$ by Lemma \ref{lem:IndEcheckEhat} and $Z(\iota(\N)) = \iota(Z(\N))$.
\end{rmk}

Let $\eta$ be a jointly cyclic and separating vector for $\iotabar(\M)$ and $\N$. Denote $j_{\iotabar(\M)} := \Ad{J_{\iotabar(\M),\eta}}$ and $j_{\N} := \Ad{J_{\N,\eta}}$ as in the previous section for the vector $\xi$ and the algebras $\iota(\N)$ and $\M$. Unlike what happens with the double conjugate morphism of $\iota$, namely $\bar\iotabar \cong \iota$, the double conjugate expectation of $E$ need \emph{not} coincide with $E$, up to unitary conjugation:

\begin{prop}\label{prop:Ebarbar}
Let $E\in \E(\M,\iota(\N))$ with finite index. Then $\overline{\Ebar} = \alpha^{-1} \circ E'' \circ \alpha$, where $\alpha := j_{\iota(\N)} \circ \iota \circ j_{\iotabar(\M)} \circ \iotabar$ is an automorphism of $\M$ mapping $\overline{\iotabar}(\N)$ onto $\iota(\N)$, and $E''$ is the bidual expectation of $E$.
\end{prop}

\begin{proof}
It follows by definition of $\overline{\Ebar}$, by the proof of Lemma \ref{lem:IndEcheckEhat} and Lemma \ref{lem:Eprimeprime}.
\end{proof}

\begin{rmk}
If $\eta$ is in the same positive cone with respect to $\iotabar(\M)$ of the vector $V\xi$, where $\iotabar = \Ad V$ is a unitary implementation of $\iotabar$ and $\xi$ is the jointly cyclic and separating vector for $\iota(\N)$ and $\M$ chosen in the previous section to define $\gamma$, then $\iotabar \circ j_{\M} = j_{\iotabar(\M)} \circ \iotabar$. In this case, by $\gamma = \iota\circ \iotabar$ and by definition of canonical endomorphism, we get $\alpha = \id_\M$ and $\overline{\Ebar} = E''$. 
\end{rmk}

Before proving our main theorem, we recall a crucial result on the Connes--Stinespring representation of $E$, not necessarily with finite index \cite[\Prop 5.1]{Lon89}, \cite[\Lem 3.3]{FiIs95}. 

\begin{prop}\label{prop:ConnesStineE}
Every $E\in\E(\M,\iota(\N))$ can be written as $E = \iota(w)^*\gamma(\slot)\iota(w)$, where $\gamma$ is a canonical endomorphism of $\M$, $w\in\N$ is an isometry such that $\iota(w):\id_{\iota(\N)}\Rightarrow\gamma_{\restriction\iota(\N)}$, and $e := \gamma^{-1}(\iota(ww^*)) \in \M_1$ is a Jones projection for $E$.
\end{prop}

Theorem \ref{thm:EEbarandconjeqns} below is a generalization of a theorem of Longo \cite[\Thm 5.2]{Lon90} from factors to arbitrary von Neumann algebras. The proof is conceptually different from the original one, since \eg $E$ need not be a conjugate expectation of $\Ebar$ by Proposition \ref{prop:bidual} and Proposition \ref{prop:Ebarbar}.
The theorem states that every conditional expectation in $\E(\M,\iota(\N))$ with finite index arises from a solution $r$, $\rbar$ of the conjugate equations \eqref{eq:catconjeqns} for $\iota:\N\to\M$ and $\iotabar:\M\to\N$ (in the concrete 2-\Cstar-category $\vNMor$ see \eqref{eq:conjeqns} below), and vice versa. 

In particular, the index of $E$ and $\Ebar$ can be expressed using  $r$, $\rbar$ (Corollary \ref{cor:Indexsolconjeqns}). The \emph{standard solutions} (in the sense of \cite{LoRo97}, \cite{GiLo19}, when they are defined) of the conjugate equations arise in this way by choosing the minimal expectation $E=E^0$ (Corollary \ref{cor:stdsol}).

\begin{thm}\label{thm:EEbarandconjeqns}
Let $\N$, $\M$ be properly infinite von Neumann algebras and let $\iota:\N\to\M$ be a morphism with conjugate morphism $\iotabar: \M\to\N$. 
For every conditional expectation $E\in \E(\M,\iota(\N))$ with finite index and with conjugate expectation $\Ebar \in \E(\N,\iotabar(\M))$, there are intertwiners
$r : \id_\N \Rightarrow \iotabar\circ\iota$ and $\rbar : \id_\M \Rightarrow \iota\circ\iotabar$, denoted by \footnote{$\N$ is denoted by a light-shaded region, $\M$ by a dark-shaded region. Diagrams should be read from right to left for 1-morphisms and from top to bottom for 2-morphisms.}:
\begin{align}
r = \tikzmath[scale=.5]{
\fill[\colN,rounded corners] (-1,-1.25) rectangle (1.5,1.25);
\draw[fill = \colM] (-.45,-1.25) -- (-.45,-.4) arc (180:0:.7) -- (.95,-1.25);
}\;,
\quad
\rbar = \tikzmath[scale=.5]{
\fill[\colM,rounded corners] (-1,-1.25) rectangle (1.5,1.25);
\draw[fill = \colN] (-.45,-1.25) -- (-.45,-.4) arc (180:0:.7) -- (.95,-1.25);
}
\end{align}
determining $E$ and $\Ebar$ as follows:
\begin{align}\label{eq:EEbarstine}
E = \frac{1}{\iota(r^*r)} \iota(r)^* \iota\iotabar (\slot) \iota(r), \quad \Ebar = \frac{1}{\iotabar(\rbar^* \iota(r^*r)\rbar)} \iotabar(\rbar)^* \iotabar\iota (r^*r \slot) \iotabar(\rbar)
\end{align}
namely:
\begin{align}
E =
\frac{1}{\tikzmath[scale=.5]{
\fill[\colN,rounded corners] (-1,-1.25) rectangle (1.5,1.25);
\clip[rounded corners] (-1,-1.25) rectangle (1.5,1.25);
\fill[\colM] (-1,-1.25) rectangle (-.5,1.25);
\draw (-.5,-1.25) -- (-.5,1.25);
\draw[fill = \colM] (.5,0) circle (.6);
}}
\tikzmath[scale=.75]{
\fill[\colN,rounded corners] (-1,-1.25) rectangle (1.5,1.25);
\clip[rounded corners] (-1,-1.25) rectangle (1.5,1.25);
\fill[\colM] (-1,-1.25) rectangle (-.5,1.25);
\draw (-.5,-1.25) -- (-.5,1.25);
\draw[fill = \colM] (1,-.25) -- (1,-.3) arc (360:180:.5) -- (0,.3) arc (180:0:.5) -- (1,.25);
\fill[white] (1.5,-.3) -- (.7,-.3) -- (.7,.3) -- (1.5,.3);
\draw (1.5,-.3) -- (.7,-.3)--(.7,.3) -- (1.5,.3);
\draw (1.1,0) node {$\slot$};
}\;,
\quad
\Ebar =
\frac{1}{\tikzmath[scale=.5]{
\fill[\colM,rounded corners] (-1,-1.25) rectangle (1.5,1.25);
\clip[rounded corners] (-1,-1.25) rectangle (1.5,1.25);
\fill[\colN] (-1,-1.25) rectangle (-.5,1.25);
\draw (-.5,-1.25)--(-.5,1.25);
\draw[fill = \colN] (.5,0) circle (.7);
\draw[fill = \colM] (.5,0) circle (.3);
}}
\tikzmath[scale=.75]{
\fill[\colM,rounded corners] (-1,-1.25) rectangle (1.5,1.25);
\clip[rounded corners] (-1,-1.25) rectangle (1.5,1.25);
\fill[\colN] (-1,-1.25) rectangle (-.5,1.25);
\draw (-.5,-1.25)--(-.5,1.25);
\draw[fill = \colN] (1,-.25)--(1,-.3) arc (360:180:.5)--(0,.3) arc (180:0:.5)--(1,.25);
\fill[white] (1.5,-.3)--(.7,-.3)--(.7,.3)--(1.5,.3);
\draw (1.5,-.3)--(.7,-.3)--(.7,.3)--(1.5,.3);
\draw[fill = \colM] (.36,0) circle (.22);
\draw (1.1,0) node {$\slot$};
}
\end{align}
and solving the conjugate equations for $\iota$ and $\iotabar$, \cf \eqref{eq:catconjeqns}:
\begin{align}\label{eq:conjeqns}
\rbar^* \iota(r) = \oneop, \quad r^* \iotabar(\rbar) = \oneop
\end{align}
namely:
\begin{align}
\tikzmath[scale=.5]{
\fill[\colM,rounded corners] (-1,-1.25) rectangle (1.5,1.25);
\clip[rounded corners] (-1,-1.25) rectangle (1.5,1.25);
\fill[\colN] (-.35,1.25) -- (-.35,-.1) arc (180:360:.3) -- (.25,1.25);
\fill[\colN] (.25,1.25) -- (.25,0) arc (180:0:.3) -- (.85,1.25);
\fill[\colN] (.25,1.25) -- (.25,0) arc (180:0:.3) -- (.85,1.25);
\fill[\colN] (.85,1.25) rectangle (1.5,-1.25);
\draw (-.35,1.25) -- (-.35,-.1) arc (180:360:.3) -- (.25,-.1) -- (.25,0) arc (180:0:.3) -- (.85,-1.25);
}\;
= 
\tikzmath[scale=.5]{
\fill[\colM,rounded corners] (-1,-1.25) rectangle (1.5,1.25);
\clip[rounded corners] (-1,-1.25) rectangle (1.5,1.25);
\fill[\colN] (.25,1.25) rectangle (1.5,-1.25);
\draw (.25,1.25) -- (.25,-1.25);
}\;,
\quad
\tikzmath[scale=.5]{
\fill[\colN,rounded corners] (-1,-1.25) rectangle (1.5,1.25);
\clip[rounded corners] (-1,-1.25) rectangle (1.5,1.25);
\fill[\colM] (-.35,1.25) -- (-.35,-.1) arc (180:360:.3) -- (.25,1.25);
\fill[\colM] (.25,1.25) -- (.25,0) arc (180:0:.3) -- (.85,1.25);
\fill[\colM] (.25,1.25) -- (.25,0) arc (180:0:.3) -- (.85,1.25);
\fill[\colM] (.85,1.25) rectangle (1.5,-1.25);
\draw (-.35,1.25) -- (-.35,-.1) arc (180:360:.3) -- (.25,-.1) -- (.25,0) arc (180:0:.3) -- (.85,-1.25);
}\;
= 
\tikzmath[scale=.5]{
\fill[\colN,rounded corners] (-1,-1.25) rectangle (1.5,1.25);
\clip[rounded corners] (-1,-1.25) rectangle (1.5,1.25);
\fill[\colM] (.25,1.25) rectangle (1.5,-1.25);
\draw (.25,1.25) -- (.25,-1.25);
}\;.
\end{align}

Vice versa, every solution $r$, $\rbar$ of the conjugate equations \eqref{eq:conjeqns} for $\iota$ and $\iotabar$ defines by the formulas in \eqref{eq:EEbarstine} a pair of conditional expectations $E$ and $\Ebar$, with $\Ebar$ conjugate of $E$.

\bigskip
In particular, the inclusion $\iota(\N)\subset\M$ has finite index (Definition \ref{def:finindex}) if and only if there is a solution of the conjugate equations \eqref{eq:conjeqns} for $\iota$ and $\iotabar$.
\end{thm}

\begin{proof}
Denote by $J_{\M_1,\xi}$ the modular conjugation of the Jones extension $\M_1$ with respect to the vector $\xi$, jointly cyclic and separating for $\iota(\N)$ and $\M$, which is also cyclic and separating for $\M_1$. Set $j_{\M_1} := \Ad J_{\M_1,\xi}$. By Proposition \ref{prop:ConnesStineE}, let $v\in\M$ be an intertwiner $v:\id_\M \Rightarrow \iota \circ \iotabar$, \ie $v m = \iota\iotabar(m) v$ for every $m\in\M$, representing $\Ehat\in\E(\M_1,\M)$ as a Connes--Stinespring dilation of a canonical endomorphism $\gamma_1: \M_1 \to \M_1$ for the inclusion $\M\subset\M_1$. Namely, $\Ehat= v^* \gamma_1(\slot) v$ on $\M_1$. The operator $v$ has the claimed intertwining properties because ${\gamma_1}_{\restriction \M} = \gamma = \iota\circ\iotabar$, as $\gamma_1 = j_\M \circ j_{\M_1} = j_{\iota(\N)} \circ j_{\M}$ on $\M_1$.
By the intertwining property, $v^*v\in Z(\M)$, but since $\Ehat(\oneop) = \oneop$, $v$ is necessarily an isometry, \ie $v^*v = \oneop$.
By the $\iota(\N)$-bimodularity of $E$ and again by the intertwining property of $v$, we have that $E(v) \iota(n) = \iota\iotabar\iota(n) E(v)$ for every $n\in\N$, thus $\iota^{-1} (E(v))\in\N$ and $\iota^{-1}(E(v)) : \id_{\N} \Rightarrow \iotabar\circ\iota$. We have to choose the correct normalization for $v$ and $\iota^{-1}(E(v))$, namely the Pimsner--Popa one, in order to produce a solution of the conjugate equations and to have the desired representation formulas for $E$ and $\Ebar$. 

First, observe that $\gamma_1$ is an isomorphism of $\M_1$ onto $\iota(\N)$ and that $vv^*\in\M$ is a Jones projection for $\Echeck = \gamma_1 \circ \Ehat \circ \gamma_1^{-1}$ by Proposition \ref{prop:ConnesStineE}.
By Lemma \ref{lem:Eprimeprime}, we have that  $(\Echeck)^{\wedge} = E''$, hence $E''(vv^*) = E(\Ind(E))^{-1}$ by Lemma \ref{lem:Eprimeinverse} and \cite[\Lem 3.1]{Kos86}, which holds for non-factor inclusions.
In terms of $E$, we have that $E(\Ind(E))^{-1} E(\Ind(E)vv^*) = E(\Ind(E))^{-1}$, thus $E(\Ind(E)vv^*) = \oneop$. Set $\vtilde := \Ind(E)^{1/2} v \in \M$, where $\Ind(E)^{1/2}$ is a positive invertible operator in $Z(\M)$, thus $E(\vtilde {\vtilde}^*) = \oneop$. Set also $w := \iota^{-1} (E({\vtilde})) \in \N$ and observe that $w:\id_\N \Rightarrow \iotabar\circ\iota$, as one can directly check using the $\iota(\N)$-bimodularity of $E$ and the fact that $\vtilde$ has the same intertwining property of $v$. With this choice of normalization we immediately have one of the conjugate equations \eqref{eq:conjeqns} solved by $w$ and $\vtilde$:
\begin{align}\label{eq:conjeq1}
w^* \iotabar(\vtilde) = \iota^{-1}(E({\vtilde}^* \iota\iotabar(\vtilde))) = \iota^{-1}(E(\vtilde{\vtilde}^*)) = \oneop.
\end{align}
We want to show that $w$ is an isometry representing $E$ as a Connes--Stinespring dilation of $\gamma$ and that the other conjugate equation in \eqref{eq:conjeqns} is solved too by $w$ and $\vtilde$. To do so, observe that for every $m\in Z(\M)$ we have $mv = vm = \iota\iotabar(m)v$, thus $mvv^* = \iota\iotabar(m)vv^*$. In particular,
$$\vtilde {\vtilde}^* = \Ind(E) vv^* = \iota\iotabar(\Ind(E))vv^*$$
thus $\vtilde {\vtilde}^*$ generates $\M$ from $\iota(\N)$ as the weakly closed span of monomials of the form $\iota(n_1) \vtilde {\vtilde}^* \iota(n_2)$, $n_1,n_2\in \N$, since the Jones projection $vv^*$ for $\Echeck$ does, \cite[\Sec 1.1.3]{PopBook}. The normal completely positive map $\iota(w)^* \iota\iotabar(\slot) \iota(w)$ defined on $\M$ is clearly $\iota(\N)$-bimodular, by the intertwining property of $w$, and by \eqref{eq:conjeq1} its value on $\vtilde {\vtilde}^*$ equals the value of $E$ on $\vtilde {\vtilde}^*$, namely $\oneop$. Thus it must coincide with $E$ and we have the desired Connes--Stinespring representation formula $E = \iota(w)^* \iota\iotabar(\slot) \iota(w)$ on $\M$. In particular, $w$ is necessarily an isometry. 
The projection $e := \gamma_1^{-1}(\iota(ww^*))\in\M_1$ is also in $\iota(\N)'$ and it implements $E$ on $\M$, \ie $e m e = E(m) e$ for every $m\in\M$, as one can show by applying $\gamma_1$ and computing
$$\iota(ww^*) \iota\iotabar(m) \iota(ww^*) = \iota(ww^* \iotabar(m) ww^*) = \iota\iotabar(\iota(w)^* \iota\iotabar(m) \iota(w)) \iota(ww^*) = \iota\iotabar(E(m)) \iota(ww^*).$$
Moreover, $e$ is a Jones projection for $E$, as $\M e \M$ contains $\oneop$, \cite[\Sec 1.1.3]{PopBook}. Indeed, $\vtilde^* e \vtilde = \oneop$ which follows from $\gamma_1(\vtilde^*e\vtilde) = \iota(\iotabar(\vtilde)^*w w^*\iotabar(\vtilde)) = \oneop$ again by \eqref{eq:conjeq1}. We have also shown that $\vtilde^*$ is a Pimsner--Popa basis consisting of one element for the inclusion $\iota(\N)\subset\M$ with respect to $E$ \cite[\Sec 1.1.4]{PopBook}, namely that $\vtilde^* e \vtilde = \oneop$ and $E(\vtilde \vtilde^*)$ is a projection, in this case $\oneop$ by our choice of normalization. 
Thus we can apply the Pimsner--Popa expansion \cite[\Sec 1.1.4]{PopBook}, \cf \cite[\Cor 5.7]{Lon90} in the subfactor case, to the elements $m\in\M$, namely $m=\vtilde^* E(\vtilde m)$. By choosing $m=\oneop$, we obtain the other conjugate equation in \eqref{eq:conjeqns} solved by $w$ and $\vtilde$:
$$\vtilde^* \iota(w) = \vtilde^* E(\vtilde) = \oneop.$$
Set $r := w$, $\rbar := \widetilde v$. In this special case where $w^*w=\oneop$ and $\vtilde^* \vtilde = \Ind(E)$, the representation formula for $\Ebar$ in \eqref{eq:EEbarstine} follows from the one of $\Ehat$, by $\Ebar = \iota^{-1} \circ \Echeck \circ \iota$ and $\Echeck = \gamma_1 \circ \Ehat \circ \gamma_1^{-1}$.

Vice versa, let $r$, $\rbar$ be a solution of \eqref{eq:conjeqns}. Observe first that $\iota(r^*r)$, $\iotabar(\rbar^*\rbar)$ and $\iotabar(\rbar^* \iota(r^*r)\rbar)$ are strictly positive invertible operators by the tensor \Cstar-categorical Pimsner--Popa bounds enforced by the conjugate equations \cite[\Lem 2.7]{LoRo97}, \cite[\Lem 8.18]{GiLo19}, and valid for non-simple tensor units. Thus $E$ and $\Ebar$ defined by \eqref{eq:EEbarstine} are well-defined normal faithful (by the Pimsner--Popa bounds) conditional expectations of $\M$ onto $\iota(\N)$ and of $\N$ onto $\iotabar(\M)$, respectively. The projection $e := \gamma_1^{-1}(\iota(r)\iota(r^*r)^{-1}\iota(r)^*)\in\M_1$ is a Jones projection for $E$. Indeed, $e\in\iota(\N)'$ and $E(m) e = e m e$ for every $m\in\M$, as one can check by applying $\gamma_1$ and using the intertwining property of $r$. Moreover, $\rbar^* \iota(r^*r)^{1/2} e \iota(r^*r)^{1/2} \rbar = \oneop$ again by the intertwining property of $r$ and by the second one of the conjugate equations \eqref{eq:conjeqns}. Similarly, $E(\iota(r^*r)^{1/2} \rbar\rbar^* \iota(r^*r)^{1/2}) = \oneop$, hence $\rbar^* \iota(r^*r)^{1/2}$ is a Pimsner--Popa basis for $\iota(\N)\subset\M$ with respect to $E$. By \cite[\Thm 3.5]{BDH88}, $E$ has finite index and its value can be computed out of the Pimsner--Popa basis (independently of its choice) as $\Ind(E) = \rbar^* \iota(r^*r) \rbar$. 
By the same theorem, \cf \cite[\Rmk 3.8]{BDH88}, the operator-valued weight from $\M_1$ onto $\M$ dual to $E$ is characterized by its $\M$-bimodularity and by taking value $\oneop$ on $e$, hence it must coincide with $F := \rbar^* \iota(r^*r) \gamma_1(\slot) \rbar$, as one can check by using the intertwining properties of $r$ and $\rbar$ and the first one of the conjugate equations \eqref{eq:conjeqns}.
We conclude that $\Ehat = (\rbar^* \iota(r^*r) \rbar)^{-1} \rbar^* \iota(r^*r) \gamma_1(\slot) \rbar$ and then $\Ebar$ defined by \eqref{eq:EEbarstine} is in fact the conjugate expectation of $E$, completing the proof.
\end{proof}

\begin{rmk}
A categorical analogue of a conditional expectation $E$ (restricted to the relative commutants in the Jones tower) is the notion of \emph{left inverse} for an object (or 1-morphism) given in \cite[\Sec 2]{LoRo97}. Left inverses are defined in the context of abstract tensor \Cstar-categories (or abstract 2-\Cstar-categories), not necessarily with simple tensor unit. In \cite[\Lem 2.5]{LoRo97}, it is shown that every left inverse has a representation formula similar to the one for $E$ in \eqref{eq:EEbarstine}, provided that $\iota$ and $\iotabar$ are conjugate in the categorical sense of the conjugate equations \eqref{eq:conjeqns}. In the previous theorem, from the analytical condition that $\Ind(E)$ is finite in $Z(\M)$, we deduce at the same time that \eqref{eq:conjeqns} admits a solution and that the formula for $E$ in \eqref{eq:EEbarstine} holds.
\end{rmk}

\begin{rmk}\label{rmk:wrongEbar}
Concerning the conjugate expectation $\Ebar$, another more \lqq symmetric" representation formula, instead of the one considered in \eqref{eq:EEbarstine}, would be $\frac{1}{\iotabar(\rbar^*\rbar)} \iotabar(\rbar)^* \iotabar\iota (\slot) \iotabar(\rbar)$. In graphical notation: 
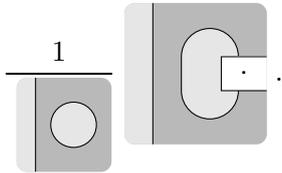
\begin{align}
\frac{1}{\tikzmath[scale=.5]{
\fill[\colM,rounded corners] (-1,-1.25) rectangle (1.5,1.25);
\clip[rounded corners] (-1,-1.25) rectangle (1.5,1.25);
\fill[\colN] (-1,-1.25) rectangle (-.5,1.25);
\draw (-.5,-1.25) -- (-.5,1.25);
\draw[fill = \colN] (.5,0) circle (.6);
}}
\tikzmath[scale=.75]{
\fill[\colM,rounded corners] (-1,-1.25) rectangle (1.5,1.25);
\clip[rounded corners] (-1,-1.25) rectangle (1.5,1.25);
\fill[\colN] (-1,-1.25) rectangle (-.5,1.25);
\draw (-.5,-1.25) -- (-.5,1.25);
\draw[fill = \colN] (1,-.25) -- (1,-.3) arc (360:180:.5) -- (0,.3) arc (180:0:.5) -- (1,.25);
\fill[white] (1.5,-.3) -- (.7,-.3) -- (.7,.3) -- (1.5,.3);
\draw (1.5,-.3) -- (.7,-.3)--(.7,.3) -- (1.5,.3);
\draw (1.1,0) node {$\slot$};
}\;.
\end{align}
By the previous theorem, this expression does not define in general a conjugate expectation of $E$, unless, \eg, 
$r^*r$ is a scalar multiple of $\oneop$.
This is the case for instance if $\N$ is a factor.

Moreover, $\rbar^* \rbar$, namely:
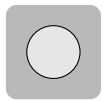
\begin{align}
\tikzmath[scale=.5]{
\fill[\colM,rounded corners] (-1,-1.25) rectangle (1.5,1.25);
\draw[fill = \colN] (.25,0) circle (.7);
}
\end{align}
belongs to $Z(\M) = \vNMor(\M,\M)(\id_\M,\id_\M)$ but it is not a good candidate expression for $\Ind(E)$, as it is not independent of the choice of the pair $r$, $\rbar$ solving the conjugate equations \eqref{eq:conjeqns} and fulfilling the first equality in \eqref{eq:EEbarstine}, \ie $E = \frac{1}{\iota(r^*r)} \iota(r)^* \iota\iotabar (\slot) \iota(r)$. Take for instance $r' := \lambda r$, $\rbar' := \lambda^{-1} \rbar$ with $\lambda > 0$.
\end{rmk}

\begin{rmk}
The relation between finiteness of the index (Definition \ref{def:finindex}) and the existence of solutions of the conjugate equations beyond subfactors has been studied with different normalizations of the equations or of the solutions in \cite[\Sec 5]{FiIs95} and \cite[\Sec 7]{BDH14}, but not the relation with a given conditional expectation together with its conjugate expectation.

Moreover, the existence of Pimsner--Popa bases consisting of a single element, for inclusions of properly infinite algebras, appears already \eg in \cite[\Sec 3]{Kos86}, \cite[\Prop 3.22]{BDH88}, \cite[\Cor 5.6]{Lon90} and \cite[\Thm 1.1.6, \Rmk 1.1.7]{PopBook}, but with no additional intertwining property on the Pimsner--Popa element as it is required in the previous theorem. See \cite[\Rmk 3.8]{DVGi18} for a discussion on this point.
\end{rmk}

From Theorem \ref{thm:EEbarandconjeqns} and its proof, and from Lemma \ref{lem:Eprimeinverse}, Remark \ref{rmk:IndEbar}, the index of $E$ and $\Ebar$ can be expressed using  $r$ and $\rbar$:

\begin{cor}\label{cor:Indexsolconjeqns}
In the assumptions of the previous theorem, we have:
\begin{align} 
\Ind(E) = \rbar^* \iota(r^*r) \rbar, \quad \Ind(\Ebar) = \frac{r^*\iotabar(\rbar^* \iota(r^*r) \rbar)r}{r^*r}
\end{align}
namely:
\begin{align}
\Ind(E) = \tikzmath[scale=.5]{
\fill[\colM,rounded corners] (-1,-1.25) rectangle (1.5,1.25);
\draw[fill = \colN] (.25,0) circle (.8);
\draw[fill = \colM] (.25,0) circle (.35);
}\;,
\quad
\Ind(\Ebar) = \frac{\tikzmath[scale=.5]{
\fill[\colN,rounded corners] (-1,-1.25) rectangle (1.5,1.25);
\draw[fill = \colM] (.25,0) circle (1.1);
\draw[fill = \colN] (.25,0) circle (.7);
\draw[fill = \colM] (.25,0) circle (.3);
}}{\tikzmath[scale=.5]{
\fill[\colN,rounded corners] (-1,-1.25) rectangle (1.5,1.25);
\draw[fill = \colM] (.25,0) circle (.38);
}}
\end{align}
independently of the choice of solutions $r$, $\rbar$ of the conjugate equations \eqref{eq:conjeqns} fulfilling \eqref{eq:EEbarstine}.
\end{cor}

\begin{rmk}\label{rmk:uniquenessunitariesrrbarfromE}
Concerning uniqueness (up to unitary 2-isomorphisms), let $r'$, $\rbar'$ be another solution of the conjugate equations \eqref{eq:conjeqns} giving rise to the same expectation $E$ via the first expression in \eqref{eq:EEbarstine}, \ie $E = \frac{1}{\iota({r'}^*r')} \iota(r')^* \iota\iotabar (\slot) \iota(r')$. 
Assume in addition that $\iota(r^*r) = \iota({r'}^*r')$, or equivalently $r^*r = {r'}^*r'$ because $\iota$ is injective by assumption, namely that $r'$, $\rbar'$ gives rise to the same (non-normalized) left inverse for $\iota$. Then by \cite[\Lem 3.3]{LoRo97}, \cite[\Lem 8.12]{GiLo19} there is a unitary $u:\iotabar \Rightarrow \iotabar$ such that $r' = u r$, $\rbar' = \iota(u) \rbar$.
\end{rmk}

\begin{cor}\label{cor:stdsol}
If in addition $\N$ and $\M$ are factors, or finite direct sums of factors and $\iota(\N)\subset\M$ is connected, \ie $Z(\iota(\N)) \cap Z(\M) = \CC\oneop$, if $E=E^0$ is the unique minimal expectation in $\E(\M,\iota(\N))$, and if $r^*r = \sqrt{[\M:\N]_0} \oneop$ in $Z(\N)$, then $r$, $\rbar$ is a standard solution \cite{LoRo97}, \cite{GiLo19} of the conjugate equations for $\iota$ and $\iotabar$ in the 2-\Cstar-category $\vNMor$.
\end{cor}

\begin{proof}
Recall that $E^0$ has scalar index, $\Ind(E^0) = [\M:\N]_0\oneop$, and that the standard solutions have equal scalar loop parameters \cite[\Prop 8.30]{GiLo19}, \cite[\Thm 2.20]{BCEGP20}. The statement follows from the previous corollary by observing that
$$[\M:\N]_0 = \|\Ind(E^0)\| = \|\rbar^* \iota(r^*r) \rbar\| = \sqrt{[\M:\N]_0} \|\rbar^* \rbar\| = \|r^* r\| \|\rbar^* \rbar\|$$
and then by using the minimality characterization of standard solutions \cite[\Thm 3.11]{LoRo97}, \cite[\Thm 8.44]{GiLo19}.
Alternatively, it follows by observing that standard solutions by their very definition \cite[\Def 8.29]{GiLo19} give $E^0$ via the first formula in \eqref{eq:EEbarstine}, \cf \cite[\Thm 2.6]{GiLo19}, they fulfill $r^*r = \sqrt{[\M:\N]_0} \oneop$, thus by the uniqueness statement in Remark \ref{rmk:uniquenessunitariesrrbarfromE}.
\end{proof}

\begin{rmk}
The finite-dimensionality assumption made on $Z(\N)$ and $Z(\M)$ in the above corollary is needed to ensure uniqueness of the minimal conditional expectation $E^0$ \cite[\Thm 2.9]{Hav90}, \cite[\Prop 3.1]{Ter92} and to have a well-defined notion of standard solution of the conjugate equations in $\vNMor$. Abstractly, standard solutions are defined so far only for rigid multi-tensor \Cstar-categories, in particular for rigid tensor categories (with simple tensor unit) \cite[\Sec 3]{LoRo97}, and for rigid 2-\Cstar-categories with finitely decomposable horizontal units \cite[\Sec 8]{GiLo19}, or equivalently \cite[\Prop 8.16]{GiLo19} with finite-dimensional 2-morphisms spaces.
\end{rmk}

\section{Q-systems}\label{sec:Qsys}

In this section, as an application of Theorem \ref{thm:EEbarandconjeqns}, we prove a generalization of \cite[\Thm 5.1]{Lon94} and \cite[\Thm 3.11]{BKLR15} from finite index (irreducible) subfactors to arbitrary finite index unital inclusions of von Neumann algebras. Theorem \ref{thm:Qsys} states that every such inclusion, together with a chosen finite index conditional expectation, can be described by a Q-system. 

\begin{defi}\label{def:uFralgebraobject}
Let $\C$ be a strict tensor \Cstar-category, or \Wstar-category with (not necessarily simple nor semisimple) tensor unit $\id$. A \textbf{unitary (or \Cstar) Frobenius algebra} in $\C$ is a triple $(\theta,x,w)$, where $\theta$ is an object in $\C$ (the algebra object), $x:\theta\otimes\theta \to\theta$ (the multiplication) and $w:\id \to \theta$ (the unit) are morphisms in $\C$, depicted as:
\begin{align}
1_\theta = \tikzmath[scale=.5]{
\draw[thick] (.25,1.25) -- (.25,-1.25);
}\;\;,
\quad
x = \tikzmath[scale=.5]{
\draw[thick] (-.25,1.25) -- (-.25,.5) arc (180:360:.5) -- (.75,1.25);
\draw[thick] (.25,0) -- (.25,-1.25);
}\;,
\quad
w = \!\! \tikzmath[scale=.5]{
\draw[thick, white] (-.25,1.25) -- (-.25,.5) arc (180:360:.5) -- (.75,1.25);
\draw[thick] (.25,0) node {$\bullet$} -- (.25,-1.25);
},
\end{align}
where $1_\theta : \theta \to \theta$ is the identity morphism associated with $\theta$, satisfying the Frobenius algebra (and coalgebra) relations:
\begin{itemize}
\item[$(i)$] $x \, (x \otimes 1_\theta) = x \, (1_\theta \otimes x)$ (associativity).
\item[$(ii)$] $x \, (w \otimes 1_\theta) = 1_\theta = x \, (1_\theta \otimes w)$ (unitality).
\item[$(iii)$] $(x \otimes 1_x) \, (1_\theta \otimes x^*) = x^* \, x = (1_\theta \otimes x) \, (x^* \otimes 1_\theta)$ (Frobenius property).
\end{itemize}
In graphical notation, $(i)$, $(ii)$ and $(iii)$ read respectively:
\begin{align}
\tikzmath[scale=.75]{
\draw[thick] (-.3,1.25) -- (-.3,.7) arc (180:360:.4) -- (.5,1.25);
\draw[thick] (.1,.3) -- (.1,.1);
\draw[thick] (.1,.1)  arc (180:360:.4) -- (.9,1.25);
\draw[thick] (.5,-.3) -- (.5,-1.25);
}\;\; = \;
\tikzmath[scale=.75]{
\begin{scope}[xscale=-1]
\draw[thick] (-.3,1.25) -- (-.3,.7) arc (180:360:.4) -- (.5,1.25);
\draw[thick] (.1,.3) -- (.1,.1);
\draw[thick] (.1,.1)  arc (180:360:.4) -- (.9,1.25);
\draw[thick] (.5,-.3) -- (.5,-1.25);
\end{scope}
}\;,
\quad\,
\tikzmath[scale=.75]{
\draw[thick] (-.25,.4) node {$\bullet$} arc (180:360:.4) -- (.55,1.25);
\draw[thick] (.15,0) -- (.15,-1.25);
}\; =\, 
\tikzmath[scale=.75]{
\draw[thick] (.25,1.25) -- (.25,-1.25);
}\;\, =\,
\tikzmath[scale=.75]{
\draw[thick] (-.25,1.25) -- (-.25,.4) arc (180:360:.4) node {$\bullet$};
\draw[thick] (.15,0) -- (.15,-1.25);
}\;,\,
\quad\,
\tikzmath[scale=.75]{
\draw[thick] (-.35,1.25) -- (-.35,-.1) arc (180:360:.4) -- (.45,0);
\draw[thick] (.05,-.5) -- (.05,-1.25);
\draw[thick] (.45,0) -- (.45,0.1) arc (180:0:.4) -- (1.25,-1.25);
\draw[thick] (.85,.5) -- (.85,1.25);
}\;\; =\;
\tikzmath[scale=.75]{
\draw[thick] (-.25,1.25) -- (-.25,.7) arc (180:360:.4) -- (.55,1.25);
\draw[thick] (.15,.3) -- (.15,-.3);
\draw[thick] (-.25,-1.25) -- (-.25,-.7) arc (180:0:.4) -- (.55,-1.25);
}\; = \;
\tikzmath[scale=.75]{
\begin{scope}[xscale=-1]
\draw[thick] (-.35,1.25) -- (-.35,-.1) arc (180:360:.4) -- (.45,0);
\draw[thick] (.05,-.5) -- (.05,-1.25);
\draw[thick] (.45,0) -- (.45,0.1) arc (180:0:.4) -- (1.25,-1.25);
\draw[thick] (.85,.5) -- (.85,1.25);
\end{scope}
}\;.
\end{align}
\end{defi}

By unitality and Frobenius property it is easy to see that $r = \rbar := x^*w$ provides a solution of the conjugate equations \eqref{eq:catconjeqns} for $\theta$ and $\bar\theta := \theta$. Thus $\theta$ is self-conjugate in $\C$ and \emph{real} in the terminology of \cite{LoRo97}.

\begin{defi}\label{def:Qsys}
If $\C$ is a subcategory of the strict tensor \Cstar-category $\End(\N)$ for a von Neumann algebra $\N$, in symbols $\C \subset \End(\N)$, and if $w^*w : \id \to \id$ namely:
\begin{align}
w^*w = \! \tikzmath[scale=.75]{
\draw[thick] (.25,.35) node {$\bullet$} -- (.25,-.35) node {$\bullet$};
}
\end{align} 
is positive invertible \footnote{This condition is automatic if $\N$ is a factor, \cf \cite[\p 147]{LoRo97}.} in $Z(\N)$, we call $(\theta,x,w)$ a \textbf{Q-system}. 
\end{defi}

In $\End(\N)$, the unitary Frobenius algebra relations read $x^2 = x \theta(x)$, $xw = \oneop = x \theta(w)$, $x\theta(x)^* = x^*x = \theta(x)x^*$. The intertwining properties read $x \theta^2(n) = \theta(n) x$, $w n = \theta(n) w$ for every $n\in\N$.

\begin{rmk}
Q-systems have been introduced in \cite{Lon94} to study finite index subfactors and to characterize the canonical endomorphism.
They have been considered first in the case of simple tensor unit $\id$, \ie for $\N$ factor, and one-dimensional space of morphisms $\id\to \theta$, \ie for irreducible subfactors $\iota(\N)\subset\M$. See \cite[\Ch 3]{BKLR15} for the study of non-irreducible extensions of $\N$, again assuming simple tensor unit, but with $\M$ not necessarily factor. In the irreducible case, Q-systems can be equivalently described by means of \emph{compact} \Wstar-algebra objects in the terminology of \cite{JoPe17}, \cite[\Sec 2.4]{JoPe20}. Q-systems are a fundamental tool \eg in the study, classification and constructions of finite index extensions of Quantum Field Theories in the algebraic setting \cite{LoRe95}, \cite{LoRe04}, \cite{KaLo04}, \cite{BKL15}, \cite{BKLR16}. For generalizations to infinite index subfactors and inclusions, which are relevant also in QFT, we refer to \cite{FiIs99}, \cite{DVGi18}, \cite{JoPe19}, \cite{BDG21}, \cite{BDG21-online}, \cite{BDG22}.
\end{rmk}

\begin{rmk}\label{rmk:invertcond}
The invertibility condition on $w^*w$ in Definition \ref{def:Qsys} has been considered also in \cite[\Def 3.3]{CHPJP21}, where it is called \emph{non-degeneracy} of the Q-system, and it has been included in the conditions of \cite[\Thm 3.11]{Mue03-I}, in the context of abstract and not necessarily unitary Frobenius algebras.
 Differently from \cite{CHPJP21} and \cite{BKLR15}, in our definition of Q-system we consider only unitary Frobenius algebras realized \eg in $\End(\N)$, and not necessarily \emph{separable} in the terminology of \cite[\Def 3.1]{CHPJP21}, \cf \cite[\Def 3.13]{Mue03-I}, \ie \emph{special} in the terminology of \cite[\Def 2.14]{GiYu20}, \cf \cite[\Def 3.2]{BKLR15}. 
\end{rmk}

Let $\iota:\N\to\M$ be a morphism between properly infinite von Neumann algebras and let $\gamma : \M \to \M$ be a canonical endomorphism of $\M$ for $\iota(\N) \subset \M$ as in Section \ref{sec:condexp}. We call $\widetilde\gamma := \iota^{-1} \circ \gamma \circ \iota : \N \to \N$ a dual canonical endomorphism of $\N$ for $\iota(\N) \subset \M$. By definition,
\begin{align}
\iota \circ \widetilde\gamma = \gamma \circ \iota.
\end{align}
Morally, $\widetilde\gamma$ is the restriction of $\gamma$ to $\iota(\N)$, and the latter is the canonical endomorphism for the one step down inclusion $\gamma(\M) \subset \iota(\N)$. Furthermore, $\gamma = \iota \circ \iotabar$ as in \eqref{eq:gammaisiotaiotabar} where $\iotabar: \M \to \N$ is a conjugate morphism of $\iota:\N\to\M$, implies that
\begin{align}\label{eq:thetaisiotabariota}
\widetilde\gamma = \iotabar \circ \iota.
\end{align}
In view of the following theorem, we shall also write $\theta$ in place of $\widetilde\gamma$.

\begin{thm}\label{thm:Qsys}
Every inclusion of properly infinite von Neumann algebras $\iota(\N)\subset \M$, with arbitrary centers, together with a given finite index conditional expectation $E\in \E(\M,\iota(\N))$, can be described by a Q-system $(\theta,x,w)$ in $\End(\N)$, where $\theta$ is a dual canonical endomorphism for the inclusion and $w$ represents the expectation as $\iota^{-1}\circ E = (w^*w)^{-1}w^* \iotabar(\slot) w$. 

Vice versa, every Q-system in $\End(\N)$ arises in this way. 
\end{thm}

\begin{proof}
Given $\iota(\N)\subset \M$, set $\theta := \iotabar \circ \iota$ to be a dual canonical endomorphism as in \eqref{eq:thetaisiotabariota}. By Theorem \ref{thm:EEbarandconjeqns}, we can choose a solution $r$, $\rbar$ of the conjugate equations \eqref{eq:conjeqns} for $\iota$ and $\iotabar$ with respect to $E$, \ie fulfilling \eqref{eq:EEbarstine}. Then $w:= r$ and $x:= \iotabar(\rbar)^* = \iotabar \otimes \rbar^* \otimes \iota$ fulfill the desired intertwining and Frobenius algebra relations in $\End(\N)$ as one can check using graphical calculus, \cf \cite[\eq (3.2.1)]{BKLR15}. The operator $w^*w = r^*r$ is invertible because $\iota(r^*r)$ is invertible by \cite[\Lem 2.7]{LoRo97}, \cite[\Lem 8.18]{GiLo19}, and $\iota$ is injective by assumption. 

Vice versa, if $(\theta,x,w)$ is a Q-system in $\End(\N)$, consider the inclusion $\theta(\N) \subset \langle\theta(\N),x\rangle$ (the von Neumann algebra generated by $\theta(\N)$ and $x$) and the map on $\langle\theta(\N),x\rangle$ defined by
$$F := \theta(w^*w)^{-1}\theta(w^*\slot w).$$
$F$ is clearly unital, completely positive, $\theta(\N)$-bimodular and normal. By the intertwining property of $x$, and by $x^2 = x \theta(x)$, $x^* = x^*x\theta(w) = x\theta(x)^*\theta(w) = x\theta(x^*w)$, we have that $\langle\theta(\N),x\rangle = x \theta(\N) = \theta(\N) x^*$ and $F(x^*x) = \theta(w^*w)^{-1}$ by $xw=\oneop$. In particular, $F$ is faithful and $F\in\E(\langle\theta(\N),x\rangle,\theta(\N))$. By direct computation one can check that $e:= \theta(w^*w)^{-1}ww^*\in\theta(\N)'\cap\N$ is a Jones projection for $F$ and that $y^*:=x\theta(w^*w)^{1/2}$ is a Pimsner--Popa basis with respect to $F$. Indeed $e$ implements $F$ on $\langle\theta(\N),x\rangle$, and $y^*ey = \oneop$, $F(yy^*) = \oneop$ hold. Thus the index of $F$ is finite and equal to $y^*y = x\theta(w^*w)x^* \in Z(\langle\theta(\N),x\rangle)$.
Moreover, $\N$ is the Jones extension of $\langle\theta(\N),x\rangle$ by $\theta(\N)$ (with respect to $F$), \ie $\N=\langle \langle\theta(\N),x\rangle, e \rangle$, for which it suffices to observe that $n = x w n = x \theta(n) w$ if $n\in\N$. Then the recognition result for the canonical endomorphism applies \cite[\Thm 4.1]{Lon90}, \cite[\Prop 2.4]{FiIs99}, and $\theta$ is a canonical endomorphism for $\N_1 := \langle\theta(\N),x\rangle \subset \N$, namely it is of the form $\theta = \Ad(J_{\N_1,\xi}J_{\N,\xi})$ on $\N$, where $J_{\N_1,\xi}$, $J_{\N,\xi}$ are the modular conjugations with respect to a jointly cyclic and separating vector $\xi$ for $\N_1$, $\N$. Let $\M:= \Ad(J_{\N_1,\xi}J_{\N,\xi})^{-1}(\N_1) = \Ad{J_{\N,\xi}}(\N_1')$ and $\iota(\N)\subset\M$ with $\iota(n) := n$. Then $\theta$ is the restriction to $\iota(\N)$ of a canonical endomorphism $\gamma := \Ad(J_{\N_1,\xi}J_{\N,\xi}) = \Ad(J_{\N,\xi}J_{\M,\xi})$ for $\iota(\N)\subset\M$ and $E := \gamma^{-1} \circ F \circ \gamma = (w^*w)^{-1}w^*\gamma(\slot)w$ is the corresponding expectation in $\E(\M,\iota(\N))$.
\end{proof}

\begin{rmk}
A \Cstar-algebraic analogue of Theorem \ref{thm:Qsys} for faithful conditional expectations with finite Watatani index \cite[\Def 1.2.2]{Wat90} (\cf the notion of strongly finite index inclusion \cite[\Def 3.6]{BDH88} and finite Pimsner--Popa bases) 
and dualizable \Cstar-correspondences \cite[\Sec 4]{KPW03} appears in \cite[\Ex 3.10, \Ex 3.12]{CHPJP21} and \cite[\Prop 4.16, \Cor 4.17]{CHPJP21}. A von Neumann-algebraic analogue for finite direct sums of $\II_1$ factors and dualizable Connes' bimodules \cite[\Sec 4]{BDH14} appears in \cite[\Sec 3]{GiYu20} and \cite[\Lem 4.1]{GiYu20}.

A categorical analogue of the second statement of Theorem \ref{thm:Qsys} concerning the splitting of the Frobenius algebra object $\theta$ as $\iotabar\circ\iota$ is due to M\"uger \cite[\Thm 3.11]{Mue03-I} in the context of abstract tensor categories, with no reference to conditional expectations. There, simplicity of the tensor unit $\id$, namely $\C(\id,\id)=\End(\id)\cong\CC$, is not necessarily assumed. It is assumed instead a weaker condition called $\End(\id)$-linearity \cite[\Def-\Prop 3.10]{Mue03-I} \footnote{Cf.\ the even weaker \lqq centrally balanced" condition of \cite[\Def 2.13]{Zit07}.}. 
If $\C\subset\End(\N)$, this condition implies that $1_\rho \otimes \lambda = \lambda \otimes 1_\rho$, or equivalently $\rho(\lambda) = \lambda$, for every $\lambda\in \End(\id) = Z(\N)$ and $\rho\in\C\subset\End(\N)$. Namely, the endomorphisms in $\C$ act trivially on the center of $\N$. 
Graphically this means that floating boxes (boxes without strands) can be commuted through every other strand.
\end{rmk}

\bigskip
\noindent
{\bf Acknowledgements.}
I am indebted to Dietmar Bisch for fruitful conversations, suggestions and questions on the topics of this paper. I also thank Roberto Longo for mentioning to me this problem and for helpful conversations.

\small

\addcontentsline{toc}{section}{References}

\newcommand{\etalchar}[1]{$^{#1}$}
\def\cprime{$'$}

\end{document}